\documentclass[journal,twoside,web]{ieeecolor}
\usepackage{generic} % IEEE-TRANS package
\usepackage{textcomp}
\usepackage{graphicx} % for pdf, bitmapped graphics files
\usepackage{subcaption}
\usepackage{mathptmx} % assumes new font selection scheme installed
\usepackage{times} % assumes new font selection scheme installed
\usepackage{amsmath} % assumes amsmath package installed
\usepackage{amssymb}  % assumes amsmath package installed
\usepackage[linesnumbered,ruled,vlined]{algorithm2e}
\usepackage{algorithmic}
\newtheorem{prop}{Proposition}

\newcommand{\rev}[2][]{\textcolor{black}{ #2}}

\title{\LARGE \bf Balanced Reduced-Order Models \\for Iterative Nonlinear Control of Large-Scale Systems}
\author{Yizhe Huang and Boris Kramer
\thanks{Y. Huang, was with the Department of Computer Science and Engineering,
        University of California San Diego. He is now with the Department of Information, Risk and Operations Management, University of Texas at Austin, TX 78712 (e-mail: yizhehuang@utexas.edu)}
\thanks{B. Kramer is with the Department of Mechanical and Aerospace Engineering, University of California San Diego, (e-mail: bmkramer@ucsd.edu)}%
}

\begin{document}
\maketitle
\thispagestyle{empty}
\pagestyle{empty}

%%%%%%%%%%%%%%%%%%%%%%%%%%%%%%%%%%%%%%%%%%%%%
\begin{abstract}
We propose a new framework to design controllers for high-dimensional nonlinear systems. The control is designed through the iterative linear quadratic regulator (ILQR), an algorithm that computes control by iteratively applying the linear quadratic regulator on the local linearization of the system at each time step. 
Since ILQR is computationally expensive, we propose to first construct reduced-order models (ROMs) of the high-dimensional nonlinear system. We derive nonlinear ROMs via projection, where the basis is computed via balanced truncation (BT) and LQG balanced truncation (LQG-BT). 
Numerical experiments are performed on a semi-discretized nonlinear Burgers equation. We find that the ILQR algorithm produces good control on ROMs constructed either by BT or LQG-BT, with BT-ROM based controllers outperforming LQG-BT slightly for very low-dimensional systems. 
\end{abstract}


%%%%%%%%%%%%%%%%%%%%%%%%%%%%%%%%%%%%%%%%%%%%%%%%%%%%%%%%%
\section{INTRODUCTION}
%%%%%%%%%%%%%%%%%%%%%%%%%%%%%%%%%%%%%%%%%%%%%%%%%%%%%%%%%
\IEEEPARstart{H}{igh}-dimensional nonlinear models are ubiquitous in science and engineering, arising, e.g., during semi-discretization of partial differential equations (PDEs), or in large-scale circuits and biological models. Real-time control and filtering for these systems remains challenging. On the one hand, using linearization and subsequently linear controllers, such as the linear quadratic regulator (LQR), leads to suboptimal control for nonlinear systems. On the other hand, while nonlinear controllers can produce better control, they quickly become computationally intractable in high dimensions. The use of nonlinear reduced-order models (ROMs) is a viable alternative in these applications. 

We propose a new framework to design controllers for high-dimensional nonlinear systems. The approach combines system-theoretic model reduction to obtain nonlinear ROMs and the iterative linear quadratic regulator (ILQR) control~\cite{li2004iterative}. ILQR computes local linearizations at every time step, and combines it with a modified, discrete-time LQR algorithm to find the optimal controller and the corresponding nominal trajectory. The next iteration then uses the new trajectory to recompute the local linearizations and repeat the process, resulting in the iterative nature of ILQR. 

The ILQR algorithm has been applied to a wide range of \textit{very low-dimensional} problems such as quadrotor UAV with cable-suspended loads~\cite{alothman2016quadrotor,alothman2017using}, lower limb exoskeleton models~\cite{sergey2016comparative}, and robotic arm models~\cite{li2004iterative}.
Moreover, several extensions to ILQR have been proposed. An iterative LQG method is developed in \cite{todorov2005generalized}, and \cite{van2014iterated} extends ILQR to systems with non-quadratic cost functions and apply it to differential-drive robots and quadrotor helicopters in environments with obstacles. The authors in \cite{banijamali2018robust} propose a way of learning local linearization of dynamics from image inputs and tests it on problems like planar systems (navigation), inverted pendulum, and cart-pole balancing. In addition, \cite{chen2017constrained} proposes constrained iterative LQR to handle the constraints in ILQR and applies it to on-road autonomous driving motion planning. 
However, while these extensions further expand the range of problems that ILQR can be applied to, none of these extensions addresses the problem that ILQR is computationally intractable for high-dimensional systems. This results from the running time of each iteration in ILQR increases quadratically with the number of dimensions.

Model reduction provides a mathematical and system-theoretic framework to reduce the dimensionality of models, see \cite{antoulas2005approximation, quarteroni2014reducedorder, benner2017modelreduction} for an overview. Here, we are interested in designing the controller using the nonlinear ROM, and using that very control on the original high-dimensional system.
One of the most popular system-theoretic model reduction technique for linear open-loop systems is the balanced truncation (BT) method \cite{moore1981principal}. It finds a state-space transformation in which the controllability and observability Gramians are equal and diagonal, and  then discards the states of the system that are hard to control and hard to observe. 
LQG balanced truncation (LQG-BT) \cite{jonckheere1983new} is a modification to BT; it produces ROMs that are suitable for closed-loop systems~\cite{king2006reduced, singler2009comparison, breiten2015feedback,benner2015lqg}. The LQG-BT algorithm does take the matrices of the control cost function into account in that it solves the LQG Riccati equations instead of the Lyapunov equations. In this work, we compare the performance of ILQR controllers based on nonlinear ROMs, which we compute via BT and LQG-BT model reduction. Our numerical testbed is the semi-discretized high-dimensional Burgers equation. 

This paper is organized as follows. Section~\ref{sec:ILQR} briefly introduces the ILQR algorithm, presents algorithmic complexity results, and the special case of quadratic-in-state systems. Section~\ref{sec:ROM} presents BT and LQG-BT model reduction. Section~\ref{sec:Numerics} numerically investigates the performance of the BT and LQG-based controllers. Section~\ref{sec:conclusions} offers conclusions and an outlook for future work.

%%%%%%%%%%%%%%%%%%%%%%%%%%%%%%%%%%%%%%%%%%%%%%%%%%%%%%%%%
\section{ITERATIVE LINEAR QUADRATIC REGULATOR} \label{sec:ILQR}
%%%%%%%%%%%%%%%%%%%%%%%%%%%%%%%%%%%%%%%%%%%%%%%%%%%%%%%%%
We define the nonlinear control problem in Section~\ref{sec:nonlinear_control_problem}. Then, in Section~\ref{sec:ILQR_alg} we introduce the ILQR method, first developed in \cite{li2004iterative}, which extends LQR to nonlinear systems. In Section~\ref{sec:costILQR}  we derive a computational cost analysis of ILQR, and in Section~\eqref{sec:ILQR_quadraticSystems} we specifically analyze quadratic-in-state systems in combination with ILQR.

%%%%%%%%%%%%%%%%%%%%%%%%%%%%%%%%%%%%%%%%%%%%%%%%
\subsection{Nonlinear Control Problem}\label{sec:nonlinear_control_problem}
We consider the nonlinear control problem of minimizing a quadratic cost subject to a nonlinear dynamical system: 
\begin{equation} \label{eq:NonlProb}
\begin{aligned}
\min_{u_0, u_1, \ldots, u_{N-1}} & \rev{J(x_{0:N}, u_{0:N-1} )}  \\
\text{s.t.} 	 \quad				& x_{k+1} = f(x_k, u_k).
\end{aligned}
\end{equation}
where $x_k \in \mathbb{R}^{n}$ is the system state at time step $k$, $u_k \in \mathbb{R}^{m}$ is the control input at time step $k$, $N$ is the number of time steps, \rev{$x_{0:N} = \{ x_0, x_1, \ldots, x_N \} $ and $u_{0:N-1} = \{ u_0, u_1, \ldots, u_{N-1}\}$ denote the sequence of states and controls, respectively,} and $f: \mathbb{R}^{n} \times \mathbb{R}^{m} \rightarrow \mathbb{R}^{n}$ is a differentiable function in both arguments that maps the current state and control to the future state. 
The cost function is assumed in quadratic form
\begin{equation}\label{eq:cost}
\rev{J(x_{0:N}, u_{0:N-1} )} = (x_N-x^*)^{\top}Q_f(x_N-x^*)+\sum_{k=0}^{N-1}(x_k^{\top}Qx_k+u_k^{\top}Ru_k),
\end{equation}
where $x^*\in \mathbb{R}^n$ is the target state. The symmetric positive semi-definite matrix $Q_f \in \mathbb{R}^{n \times n}$ defines the final state cost, and the symmetric positive semi-definite matrix $Q \in \mathbb{R}^{n \times n}$ defines the intermediate state cost, while the positive definite matrix $R \in \mathbb{R}^{m \times m}$ defines the control cost.

%%%%%%%%%%%%%%%%%%%%%%%%%%%%%%%%%%%%%%%%
\subsection{The ILQR Algorithm}\label{sec:ILQR_alg}
The ILQR algorithm provides a solution for the nonlinear control problem in \eqref{eq:NonlProb}. The iterative algorithm starts with a nominal control sequence \rev{$u_{0:N-1}$} and nominal trajectory \rev{$x_{0:N}$} that results from that input. The algorithm then improves \rev{$u_{0:N-1}$} at each iteration so to minimize~\eqref{eq:cost} while using local linearizations at every time step. We refer to \cite{li2004iterative} for details. 
ILQR iteratively computes a local linearization 
\begin{equation}\label{eq:state}
\delta x_{k+1} = A_k \delta x_k + B_k \delta u_k
\end{equation}
given a nominal control sequence \rev{$u_{0:N-1}$} and corresponding trajectory \rev{$x_{0:N}$}, where $\delta x_k$ and $\delta u_k$ denote the deviation from the nominal trajectory, while $A_k = D_xf(x_k, u_k)$ and $B_k = D_uf(x_k, u_k)$ are the Jacobians of $f$ with respect to $x$ and $u$, respectively.
The algorithm then uses the local linearization in a modified LQR algorithm to compute a local optimal control, which is
\begin{equation}\label{eq:computedeltau}
\delta u_k = -\rev{K_k}\delta x_k - \rev{K_k}^v v_{k+1} - \rev{K_k^u}u_k,
\end{equation}
where $\rev{K_k^v} = (B^{\top}_kS_{k+1}B_k+R)^{-1}B^{\top}_k$ and $\rev{K_k^u} = (B^{\top}_kS_{k+1}B_k+R)^{-1}R$, and where 
\begin{align*}
\rev{K_k} &= (B^{\top}_kS_{k+1}B_k+R)^{-1}B^{\top}_kS_{k+1}A_k \\
\rev{S_k} &= A^{\top}_kS_{k+1}(A_k-B_kK_k)+Q \\
\rev{v_k} &= (A_k-B_kK_k)^{\top}v_{k+1}-K_k^{\top}Ru_k+Qx_k.
\end{align*}
Observe that to compute $S_k$ and $v_k$ we need $S_{k+1}$ and $v_{k+1}$, yielding a backward-in-time iteration. In addition, we have that $S_N = Q_f$, $v_N = Q_f(x_N - x^*)$. This control is then used to compute the nominal trajectory in the next iteration, and the process proceeds until convergence. 
Algorithm~\ref{alg:ILQR} summarizes the ILQR algorithm. We terminate the algorithm when the relative difference of cost between two consecutive iterations is small, $| J_{\text{old}} - J | / J_{\text{old}} \leq \text{tol}$. 

\begin{algorithm}\caption{ILQR Algorithm}\label{alg:ILQR}
	\SetAlgoLined
	\textbf{Input}: Nonlinear term $f(x_k, u_k)$, initial state $x_0$, target $x^*$, initial control sequence \rev{$u_{0:N-1}$}, number of time steps $N$, matrices in cost function $Q$, $Q_f$, $R$, convergence threshold \texttt{tol}\;
	\textbf{Output}: Control \rev{$u_{0:N-1}$}, number of ILQR iterations $l$\;
	$l = 0$, $S_N = Q_f$, $\delta x_0 = 0$\;
	$x_{k+1} = f(x_k, u_k)$, $\forall k = 0, 1, \ldots, N-1$\;
	$J_{\text{old}}=1, \ J=1+2\texttt{tol}$\;
	\While{$|J_{\text{old}} - J| / J_{\text{old}} > \texttt{tol}$} {
		$J_{\text{{old}}} = J $\;
		$J=0$\;
		$v_N = Q_f(x_N - x^*)$\;
		$x_{1:N}^{\text{{old}}} = x_{1:N}$\;		
		\For{$ k \gets 0$ \KwTo $N-1$}{
			$A_k = D_xf(x_k, u_k)$, $B_k = D_uf(x_k, u_k)$
		}
		\For{$ k  \gets N-1$ \KwTo 0}{
			$\rev{K_k} = (B^{\top}_kS_{k+1}B_k+R)^{-1}B^{\top}_kS_{k+1}A_k$\;
			$\rev{K_k^v} = (B^{\top}_kS_{k+1}B_k+R)^{-1}B^{\top}_k$\;
			$\rev{K_k^u} = (B^{\top}_kS_{k+1}B_k+R)^{-1}R$\;
			$S_k = A^{\top}_kS_{k+1}(A_k-B_k \rev{K_k})+Q$\;
			$v_k = (A_k-B_k\rev{K_k})^{\top}v_{k+1}-\rev{K_k^{\top}}Ru_k+Qx_k$\;
		}
		\For{$k \gets 0$ \KwTo $N-1$}{
			$\delta u_k = -\rev{K_k}\delta x_k - \rev{K_k^v} v_{k+1} - \rev{K_k^u}u_k$\;
			$u_{k+1} = u_k + \delta u_k$\;
			$x_{k+1} = f(x_k, u_k)$\;
			$\delta x_{k+1} = x^{\text{{old}}}_{k+1} - x_{k+1}$\;
			$J = J + x_{k+1}^{\top}Qx_{k+1} + u_{k+1}^{\top}Ru_{k+1} $\;
		}
		$J = J + x_N^{\top}Q_fx_N $, \ $l = l+1$\;
	}
\end{algorithm}

%%%%%%%%%%%%%%%%%%%%%%%%%%%%%%%%%%%%%%%%%%%%%%%%%%%%%%%%%
\subsection{Computational Cost of ILQR} \label{sec:costILQR}
In each iteration, Algorithm~\ref{alg:ILQR} completes three portions. First, we obtain the linearizations (step 11-13); second, we compute $\rev{K_k, K^v_k, K^u_k}, S_k, v_k$ (step 14-20); third, we compute the control increment \rev{$(\delta u)_{0:N-1}$} and the new trajectory \rev{$x_{0:N}$} (step 21-27).
Linearizations are obtained in $N$ Jacobian calls if it is known analytically (e.g., for polynomial models). If we need to approximate the Jacobian numerically, then we have to march the system equation once for each dimension of state and control, which requires evaluating $f(x_k, u_k)$ at total of $N(n+m)$ times.
The second and third portion require matrix multiplication and inversion. State-of-the-art algorithms of those operations have complexity $\mathcal{O}(n^{2.4})$\cite{COPPERSMITH1990251}, whereas classical methods scale as $\mathcal{O}(n^3)$, so the complexity of the second and third pass is at best $\mathcal{O}(N(n^{2.4}+m^{2.4}))$. 
 
\begin{prop} \label{prop:complexity}
	The complexity of the $l$th iteration of the ILQR Algorithm~\ref{alg:ILQR} is $\mathcal{O}(N(n^{2.4}+m^{2.4}))$, where $n$ is the state dimension, $m$ is the number of controls, and $N$ is the number of time steps.
\end{prop}

The previous result analyzes  a single iteration of the ILQR algorithm. We note that the number of iterations of the ILQR algorithm $l$ varies significantly, depending on the structure of the nonlinear function $f(x_k,u_k)$, as we see next.
\begin{prop}
	When applied to a discrete-time (DT) LTI system, the ILQR algorithm converges in one iteration. The resulting control is the same as the output of DT-LQR\footnote{Refer to, e.g., \cite[Thm. 6.28]{Kwakernaak1972LinearOC} for DT-LQR control.}.
\end{prop}

%%%%%%%%%%%%%%%%%%%%%%%%%%%%%%%%%%%
\subsection{ILQR for Quadratic-in-State Systems} \label{sec:ILQR_quadraticSystems}
The Euler equations in specific volume form describe a purely quadratic system and are commonly used in aerospace applications. Other models, such as Navier-Stokes and Burgers equations are nearly quadratic for small viscous terms. Thus, we analyze the ILQR algorithm specifically for quadratic-in-state systems of the form:
\begin{equation}
 x_{k+1} = G(x_k\otimes x_k) + Bu_k, \label{eq:quadraticEq}
\end{equation}
where $G\in \mathbb{R}^{n\times n^2}$, $B\in \mathbb{R}^{n\times m}$ and  $\otimes$ denotes the Kronecker product, defined as $ x \otimes x =[ x_k^{(1)} x_k^\top, \ldots, x_k^{(n)} x_k^\top]^\top \in \mathbb{R}^{n^2}$.
\begin{prop}
	Let $Q = 0$ in the cost equation \eqref{eq:cost} (penalty only on the final state). Suppose the initial state is $x_0 = 0$, and the initial control sequence is \rev{$u_{0:N-1}$}.
	Then the ILQR algorithm applied to the quadratic system~\eqref{eq:quadraticEq} takes at least $N$ iterations to converge, where $N$ is the number of time steps.
\end{prop}
\begin{proof} 
Since $x_0 =0$ and initial control $u_k =0, k=1, 2, \ldots$, the initial trajectory of the system~\eqref{eq:quadraticEq} is stationary: $x_k = 0$ for $k=0,1,\ldots, N-1$. Likewise, in the local linearization $A_k = D_x(G(x_k\otimes x_k) + Bu_k) = G(x_k\otimes I_n + I_n \otimes x_k)$, therefore, $A_k = 0$ in all local linearizations and $B_k = D_u(G(x_k\otimes x_k) + Bu_k) = B$.
We first analyze the behavior of the ILQR algorithm in the first iteration. There, except $S_N$, all of $S_k=0$, since $A_k = Q = 0$. Therefore, with the exception of step $N-1$, we have $K = 0$, $K_v = R^{-1}B^{\top}$ and $K_u = I$. Therefore $(A_k - B_kK) = 0 \ \forall k$, which means that all $v_k=0$ except $v_N$. Therefore, $\delta u_k = 0$ except $\delta u_{N-1}$. In particular, $\delta u_{N-1} = (B^{\top}Q_fB+R)^{-1}B^{\top}Q_fx^*$. As a result, only $u_{N-1}$ is updated in the first iteration.
Similarly, we find that only $u_{N-2}$ and $u_{N-1}$ are updated in the second iteration, and so on. It takes $N$ iterations before $u_0$ gets updated; thus, it takes at least $N$ iterations for ILQR algorithm to converge.
\end{proof}

The previous theorem points to a computational  challenge. If it takes $O(N)$ iterations to converge and since in each iteration the time complexity is $O(N)$ with respect to the number of time steps (see Proposition~\ref{prop:complexity}) the total time complexity is $O(N^2)$. For stiff problems and long simulation times this becomes intractable. 
To address this problem, it may be helpful to use a coarse time step on the first few iterations and then use finer timer step later to speed up convergence. For example, say $N=10$ for the first few iterations, then all $u_k$ can be updated in $10$ iterations, and we can later switch to $N = 1000$. This strategy can be thought of as using the output from coarse ILQR to initialize fine-grain ILQR. 
Clearly, the use of ROMs also speeds up the algorithm since the time needed for each iteration is reduced significantly.

%%%%%%%%%%%%%%%%%%%%%%%%%%%%%%%%%%%%%%%%%%%%%%%%%%%%%%%%%
\section{SYSTEM-THEORETIC MODEL REDUCTION} \label{sec:ROM}
%%%%%%%%%%%%%%%%%%%%%%%%%%%%%%%%%%%%%%%%%%%%%%%%%%%%%%%%%
We present two system-theoretic model reduction techniques: balanced truncation~\cite{moore1981principal} in Section~\ref{sec:BT}, and LQG-BT~\cite{jonckheere1983new} in Section~\ref{sec:LQGBT}. %For more details, we refer to the thesis~\cite{huang2020thesis}.
Both methods are defined for continuous-time LTI systems $\dot{x} = Ax + Bu, \ y = Cx$ where $y \in \mathbb{R}^{p}$ is the output and the system matrices are $A \in \mathbb{R}^{n \times n}$, $B \in \mathbb{R}^{n \times m}$ and $C \in \mathbb{R}^{p \times n}$. 
We use these methods for discrete-time nonlinear systems as follows. First, we linearize the nonlinear systems and then convert the  system from discrete-time to continuous-time using a bilinear Tustin transformation. We then use BT and LQG-BT to compute the state-space transformations, and use those transformations to reduce the nonlinear systems.
The control inputs keep the system around the equilibrium/linearization point. Thus, using these methods to obtain the projection matrices for the nonlinear ROM is more appropriate than in the open-loop nonlinear setting. More ROM details are given in Section~\ref{sec:NumROM}.
Both concepts are for infinite-time horizon. There are finite-time BT methods~\cite{kurschner2018balanced}, yet we found the infinite-time BT method to outperform the finite-time BT on our test problems.

%%%%%%%%%%%%%%%%%%%%%%%%%%%%%%%%%%%%%%%%%%%%%%
\subsection{Computing the Balanced Truncation Basis}\label{sec:BT}
Balanced truncation computes a state-space transformation for stable systems in such a way that the controllability and observability Gramians of the ROM are equal and diagonal, e.g., only the easy-to-control and easy-to-observe states are kept in the ROM. The controllability Gramian $P$, and the observability Gramian $Q$, satisfy the Lyapunov equations
\begin{equation*}
AP + PA^\top+BB^\top = 0, \qquad  QA + A^{\top}Q+C^{\top}C = 0. 
\end{equation*}
Applying the transformation $x = T\tilde{x}$ to the LTI system yields $\dot{\tilde{x}} = \tilde{A}\tilde{x} + \tilde{B}u, \ \tilde{y} = \tilde{C}\tilde{x}$
where $\tilde{A} = T^{-1}AT$, $\tilde{B} = T^{-1}B, \tilde{C} = CT$. The Gramians of the transformed system are $\tilde{P} = T^{-1}PT^{-\top}$ and $\tilde{Q} = T^{\top}QT$.
We want $T$ such that $\tilde{P} = \tilde{Q}= \Sigma$, hence
\begin{equation*}
\Sigma^2 = \tilde{P}\tilde{Q} = T^{-1}PT^{-\top}T^{\top}QT = T^{-1}PQT.
\end{equation*}
Therefore, $T$ can be obtained by computing the eigenvectors of $PQ$ and eigenvalues $\sigma_i^2$ in $\Sigma^2 = \text{diag}(\sigma_1^2, \ldots, \sigma_n^2)$. To ensure $\tilde{P} = \tilde{Q}$, proper scaling of the eigenvectors is required.
The complete balanced truncation method for stable systems in an efficient implementation is given in Algorithm~\ref{alg:BT}. This implementation only computes the leading $r$ columns of $T_r$, and $T_{l}$ and is therefore numerically more stable and efficient. 
\begin{algorithm}[tb!] 
	\caption{Balanced truncation algorithm} \label{alg:BT}
	\begin{algorithmic}[1]
	\REQUIRE Matrices $A \in \mathbb{R}^{n \times n}$, $B \in \mathbb{R}^{n \times m}$ and $C \in \mathbb{R}^{p \times n}$, \\  \ \ \quad ROM dimension $r$
	\ENSURE Transformation matrices $T_r , T_{l}$
	\STATE Solve $AP + PA^\top+BB^\top = 0$ for $P=RR^\top$
	\STATE Solve $QA + A^{\top}Q+C^{\top}C = 0$ for $Q=LL^\top$
	\STATE $r = \min(r, \text{rank}(L), \text{rank}(R))$
	\STATE $U, \Sigma, V = \texttt{svd}(L[:,1:r]^\top R[:,1:r])$
	\STATE $T_r = R[:,1:r]V \Sigma^{-\frac{1}{2}}$
	\STATE $T_{l} = \Sigma^{-\frac{1}{2}}U^\top L[:,1:r]^\top$
	\end{algorithmic}
\end{algorithm}

%%%%%%%%%%%%%%%%%%%%%%%%%%%%%%%%%%%%%%%%%%%%%%%%%%%%%
\subsection{Computing the LQG-Balanced Truncation Basis}\label{sec:LQGBT}
LQG-BT \cite{jonckheere1983new} applies BT to closed-loop (stabilized) systems, and therefore also applies to unstable systems where, however, $(A,B)$ is stabilizable and $(A,C)$ is detectable. It balances the unique positive definite solutions $P,Q$ of the LQG algebraic Riccati equations
\begin{align*}
AP + PA^{\top} - PC^{\top} CP + BB^{\top} &= 0, \\
A^{\top}Q + QA - QBR^{-1}B^{\top} Q + C^{\top}C &= 0.
\end{align*}
LQG-BT follows the same steps as Algorithm~\ref{alg:BT}, except at Steps 1--2 it uses the solutions to the LQG AREs. As a byproduct, we directly obtain LQR or LQG controllers.

%%%%%%%%%%%%%%%%%%%%%%%%%%%%%%%%%%%%%%%%%%%%%%%%%%%%%%%%%
\section{NUMERICAL RESULTS: BURGERS EQUATION} \label{sec:Numerics}
%%%%%%%%%%%%%%%%%%%%%%%%%%%%%%%%%%%%%%%%%%%%%%%%%%%%%%%%%
Burgers equation has been widely used in modeling of fluids and traffic flows, and as a numerical testbed for nonlinear model reduction and control~ \cite{kunisch1999control,atwell2001reduced,kunisch2005pod,kramer2016model}. 

%%%%%%%%%%%%%%%%%%%%%%%%%%%%%%%%%%%%%%%%%%%%%%%%%%%%%%%%%
\subsection{PDE Model and Semi-Discretized Model}
We consider a 1D Burgers equation with the setup as in \cite{borggaard2020quadratic}. The PDE model is
\begin{equation*}%\label{eg:BurgersPDE}
\dot{z}(\xi, t) = \epsilon z_{\xi\xi}(\xi, t) - \frac{1}{2}(z^2(\xi, t))_\xi + \sum_{k=1}^m \chi_{[(k-1)/m,k/m]}(\xi) u_k(t)
\end{equation*}
for $t>0$ and where $\xi$ is the spatial variable, $u_k(t)$ is the distributed control, and $m$ is the dimension of control, $z \in H_{\text{per}}^1(0,1)$, which means that the system has periodic boundary conditions. Moreover, $z(\xi, 0) = z_0(\xi) = 0.5 \sin(2\pi \xi)^2$ for $\xi \in [0,0.5]$ and zero otherwise. The function $\chi_{[a,b]}(x)$ is the characteristic function over $[a,b]$. Thus, the spatial domain is divided into $m$ equidistant intervals, and each control is applied on one corresponding interval.
When discretized with linear finite elements, and after an inversion with the mass matrix, the semi-discretized system is a nonlinear ODE which is quadratic in the state, and linear in the control:
\begin{equation}
\dot{x} =  A x + G (x \otimes x) + B u, \qquad y = Cx, \label{eq:BurgersODE}
\end{equation}
where $G \in \mathbb{R}^{n\times n^2}$, and initial condition $x(0)=x_0$.  
We choose the following parameters for our numerical experiments: $n=101$ states, $m=5$ controls, and viscosity $\epsilon = 5 \times 10^{-3}$ to make the nonlinear quadratic term dominant. The system is simulated for $t_f = 5$s, and the time interval is divided into $N = 500$ time steps. The trajectory is computed using the backward Euler method, which is implemented using the Levenberg-Marquardt algorithm. The output matrix $C = I_n$, the state cost matrices $Q = Q_f = C^{\top}C$, and the input cost matrix is $R = 10^{3} \times I_m$, where $I_n$ is the $n \times n$ identity matrix.

%%%%%%%%%%%%%%%%%%%%%%%%%%%%%%%%%%%%%%%%%%%%%%%%%
\subsection{BT and LQG-BT Reduced-Order Models} \label{sec:NumROM}
We linearize \eqref{eq:BurgersODE} around the zero equilibrium and use the linearized matrices $A_\text{lin}, B, C$ in the BT or LQG-BT Algorithm~\ref{alg:BT} to compute the transformations $T_r, T_{l}$. The linearized matrix $A_\text{lin}$ is stable. We then project the full-order model (FOM) in \eqref{eq:BurgersODE} to obtain
\begin{equation} \label{eq:BurgersROM}
\dot{x}_r = A_r x_r + G_r (x_r \otimes x_r) + B_r u_r, \qquad   y_r = C_rx_r
\end{equation}
where $A_r = T_{l}AT_r$, $B_r = T_{l}B$, $C_r = CT_r$, $G_r = T_{l} G (T_r \otimes T_r)$ and $x_r(0)= T_{l}x_0$. 
Fig.~\ref{fig:Hankel_svals} shows the normalized singular values of the matrix $L^{\top}R$ in the balancing procedure (see step~4 in Algorithm~\ref{alg:BT}) for both BT and LQG-BT. For the standard BT algorithm, those are also known as the Hankel singular values. We observe that the first $20$ Hankel singular values decay fast, and all Hankel singular values except the first one come in pairs. In contrast, the normalized singular values in the LQG-BT decay much more slowly, and only level off at $r>60$. Similar to the Hankel singular values, these singular values except the first one come in pairs. The fact that the singular values in the BT algorithm decay much faster indicates that it is likely that BT outperforms LQG-BT in open-loop simulation accuracy, as when we truncate the transform matrix we are discarding dimensions that have lower associated singular values. 
\begin{figure}[tb]	\centering
	\includegraphics[width=0.4\textwidth]{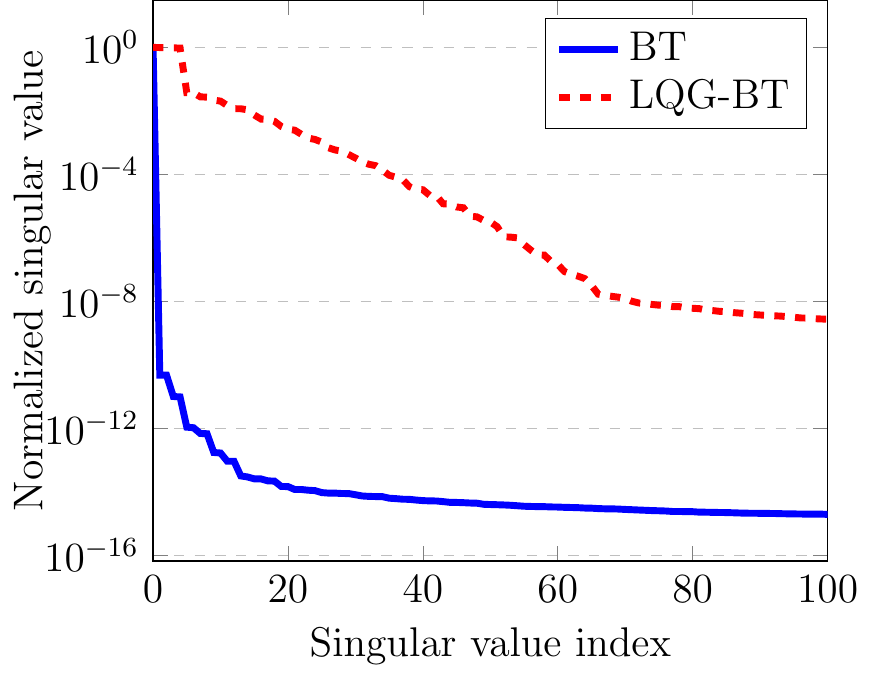}
	\caption{Normalized singular values of the matrix $L^{\top}R$. For balanced truncation, those are the Hankel singular values.}
	\label{fig:Hankel_svals}
\end{figure}

%%%%%%%%%%%%%%%%%%%%%%%%%%%%%%%%%%%%%%%%%%%%%%%%%%%%%
\subsection{Performance of ROM-based ILQR Controllers}
We use the ROMs to compute controllers $u_r^{\text{BT}}$ and $u_r^{\text{LQGBT}}$ via the ILQR Algorithm~\ref{alg:ILQR} with \texttt{tol}=$3\times10^{-5}$. We note that the Jacobians of \eqref{eq:BurgersROM}, which are used to obtain the local linearization of the system in the ILQR algorithm, are $D_{x_r}\dot{x}_r = A_r + G_r (x_r \otimes I_{n_r} + I_{n_r} \otimes x_r)$ and $D_{u_r}\dot{x}_r = B_r$, where $n_r$ is the number of dimensions of the state $x_r$ in the ROM. The availability of the Jacobians in explicit form significantly speeds up the computational routine, and should be used whenever available, c.f. Proposition~\ref{prop:complexity}.

Fig.~\ref{fig:FOM_ROM_Cost} shows the ROM cost $J(x_r, u_r)$ computed as in~\eqref{eq:cost} yet with state $x_r$, controls $u_r$ and cost matrix $Q_r = Q_{f,r} = C_r^\top C_r$, versus the iterations of the ILQR algorithm. Moreover, the figure shows the FOM cost $J(x, u_r)$ from \eqref{eq:cost}  when the  ROM controller is used for the FOM (thin horizontal lines). Both ROMs had dimension $r=5$. Note that the ILQR algorithm converges monotonically for both models terminating requiring 170 iterations with the LQG-BT model and 223 iterations with the BT model. With the final controllers the cost function evaluated for the FOM, $J(x, u_r)$, is 68.9 for BT and 63.6 for LQG-BT showing a slightly better performance of LQG-BT. The norm of the controls obtained from ILQR is  shown in Fig.~\ref{fig:control_norm}, and Fig.~\ref{fig:controlled_FOM_output} shows the corresponding output. Due to the periodic boundary conditions and the large penalty on the control, the state shows a low-magnitude periodic behavior. 

Fig.~\ref{fig:FOMSurface} shows the open-loop behavior of the system when excited with the initial condition $x_0$  and no inputs. The initial condition gets convected in a nonlinear fashion to the right. In contrast, the controlled systems with the ROM-based ILQR controllers are shown in Fig.~\ref{fig:FOMwBTcontrol} for BT and Fig.~\ref{fig:FOMwLQGcontrol} for LQG-BT. In both cases, the control signal (Fig.~\ref{fig:control_norm}) reduces the initial condition significantly, showing the efficacy of the control.

\begin{figure}[tb] \centering
	\includegraphics[width=0.4\textwidth]{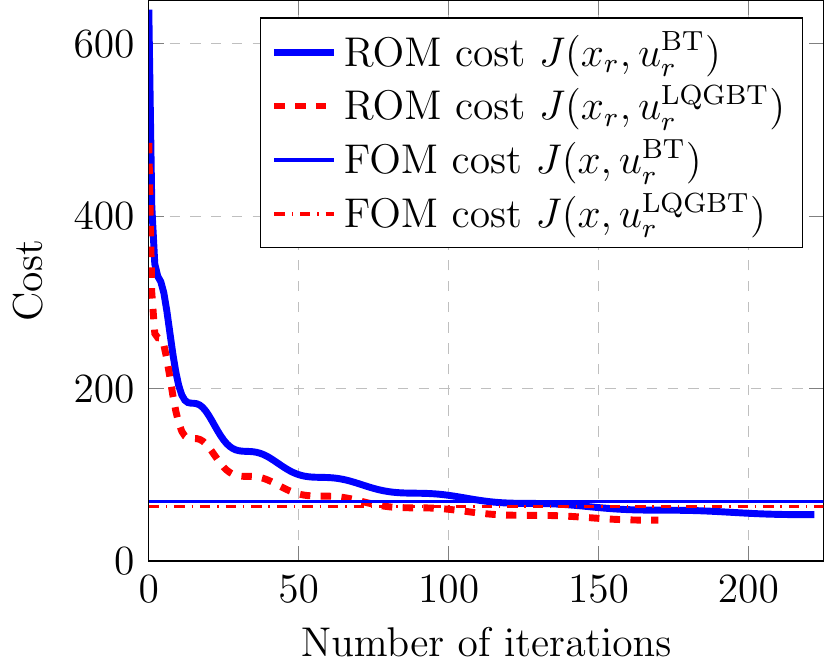}
	\caption{ROM cost vs number of iterations in the ILQR algorithm for $r=5$. The FOM cost when using the converged ILQR controller shown in thin lines, which is 68.9 for BT and 63.6 for LQG-BT.}
	\label{fig:FOM_ROM_Cost}
\end{figure}

\begin{figure}[tb] \centering
	\includegraphics[width=0.4\textwidth]{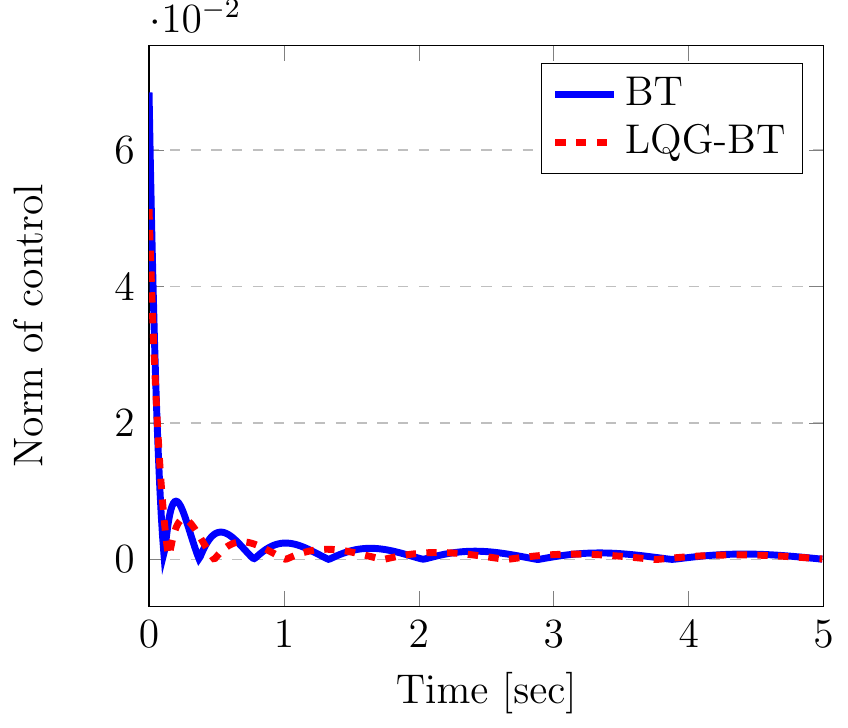}
	\caption{Norm of ILQR control $\Vert u \Vert_2$ computed using the two nonlinear ROMs with $r=5$.}
	\label{fig:control_norm}
\end{figure}

\begin{figure}[tb] \centering
	\includegraphics[width=0.4\textwidth]{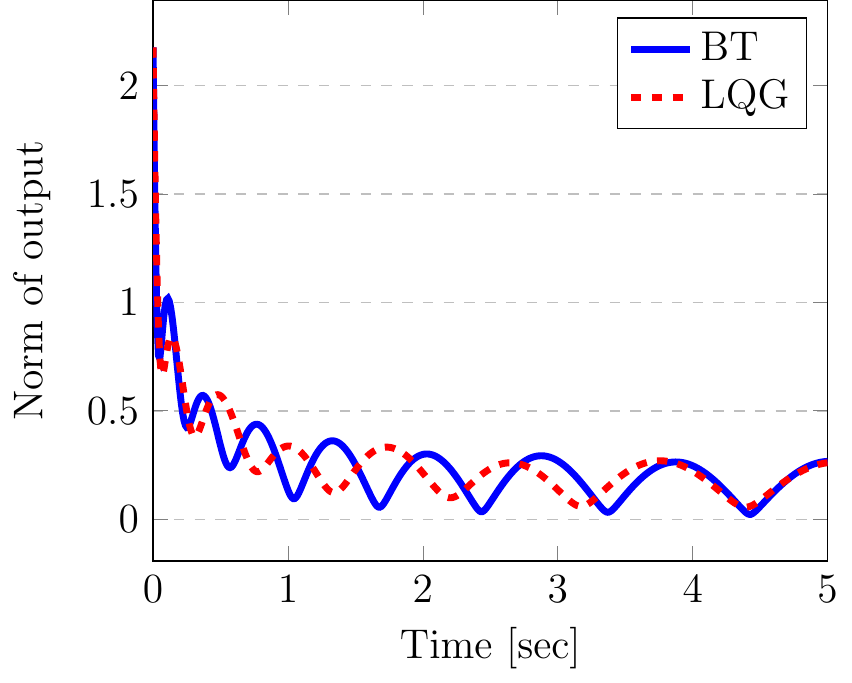}
	\caption{Norm of outputs $\Vert y \Vert_2$ in the closed-loop FOM, with the ILQR controller computed using the two nonlinear ROMs with $r=5$.}
	\label{fig:controlled_FOM_output}
\end{figure}

\begin{figure}[tb] \centering
	\includegraphics[width=.45\textwidth,trim=0 0.5cm 0 1.8cm, clip]{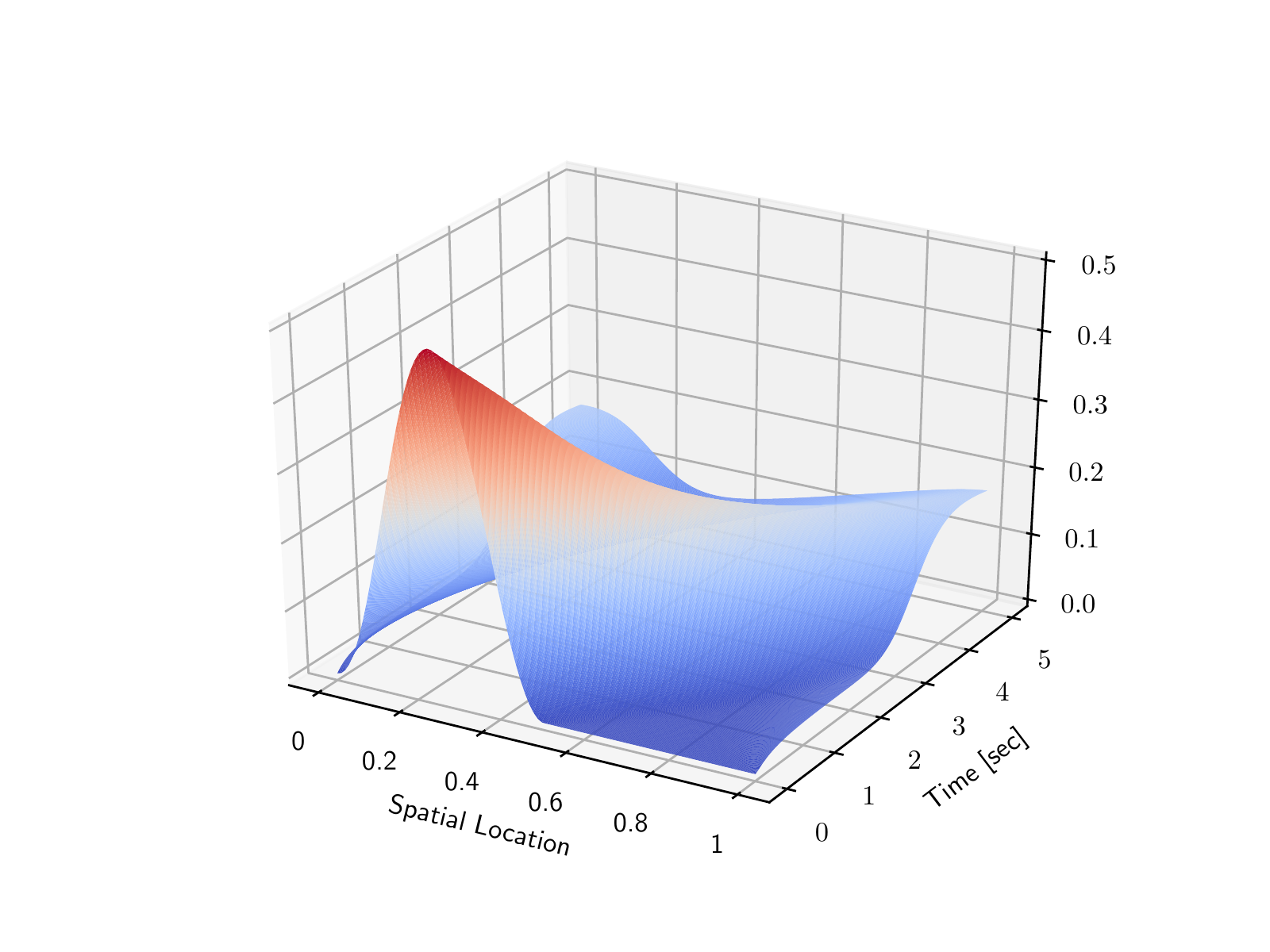}
	\caption{FOM simulation: initial condition $x_0$ and zero input.}
	\label{fig:FOMSurface}
\end{figure}

\begin{figure}[tb] \centering
	\begin{subfigure}{.45\textwidth}
		\includegraphics[width=\textwidth, trim=0 0.5cm 0 1.8cm, clip]{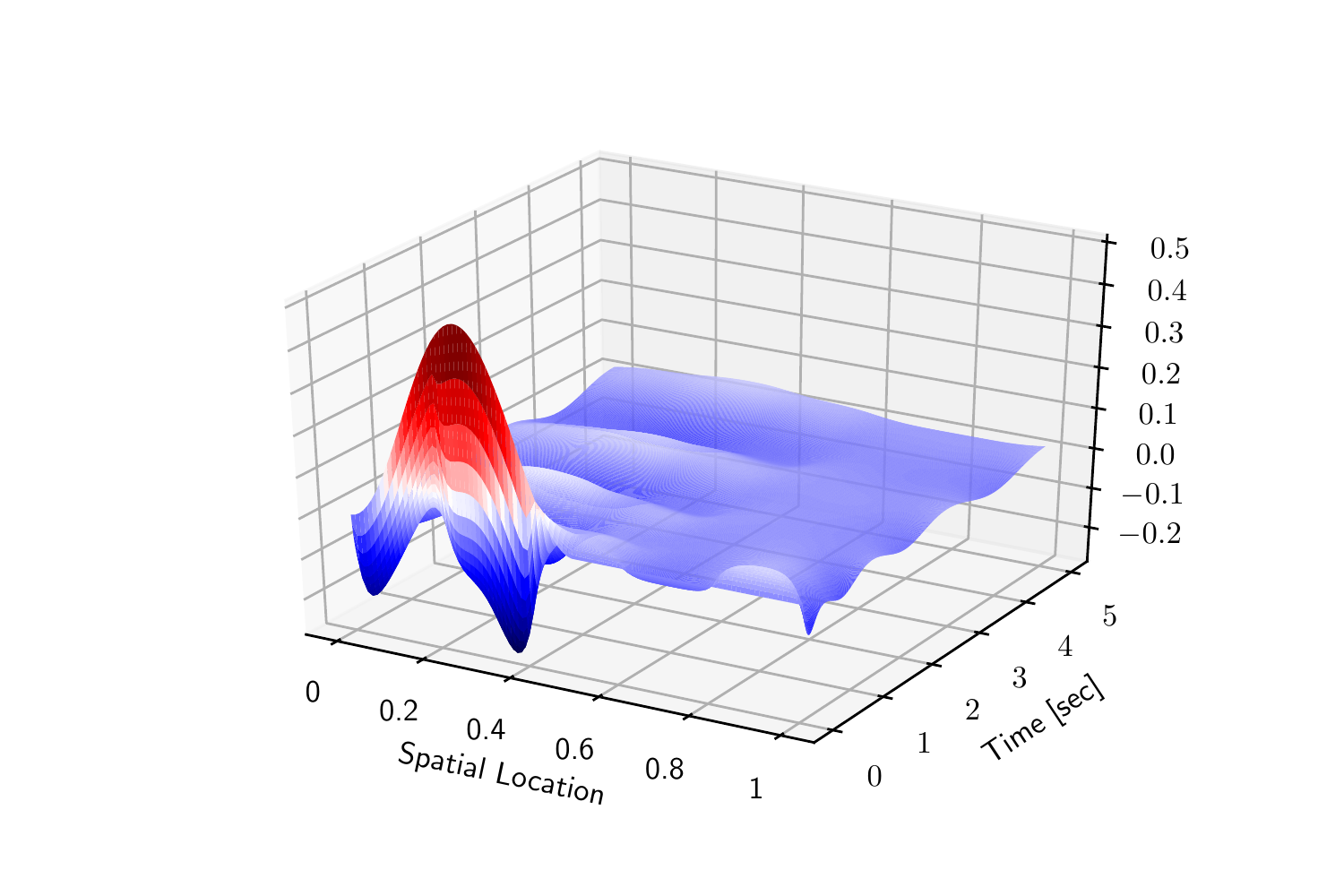}
		\caption{Standard BT}
		\label{fig:FOMwBTcontrol}
	\end{subfigure}
	\begin{subfigure}{.45\textwidth}
		\includegraphics[width=\textwidth,trim=0 0.5cm 0 1.8cm, clip]{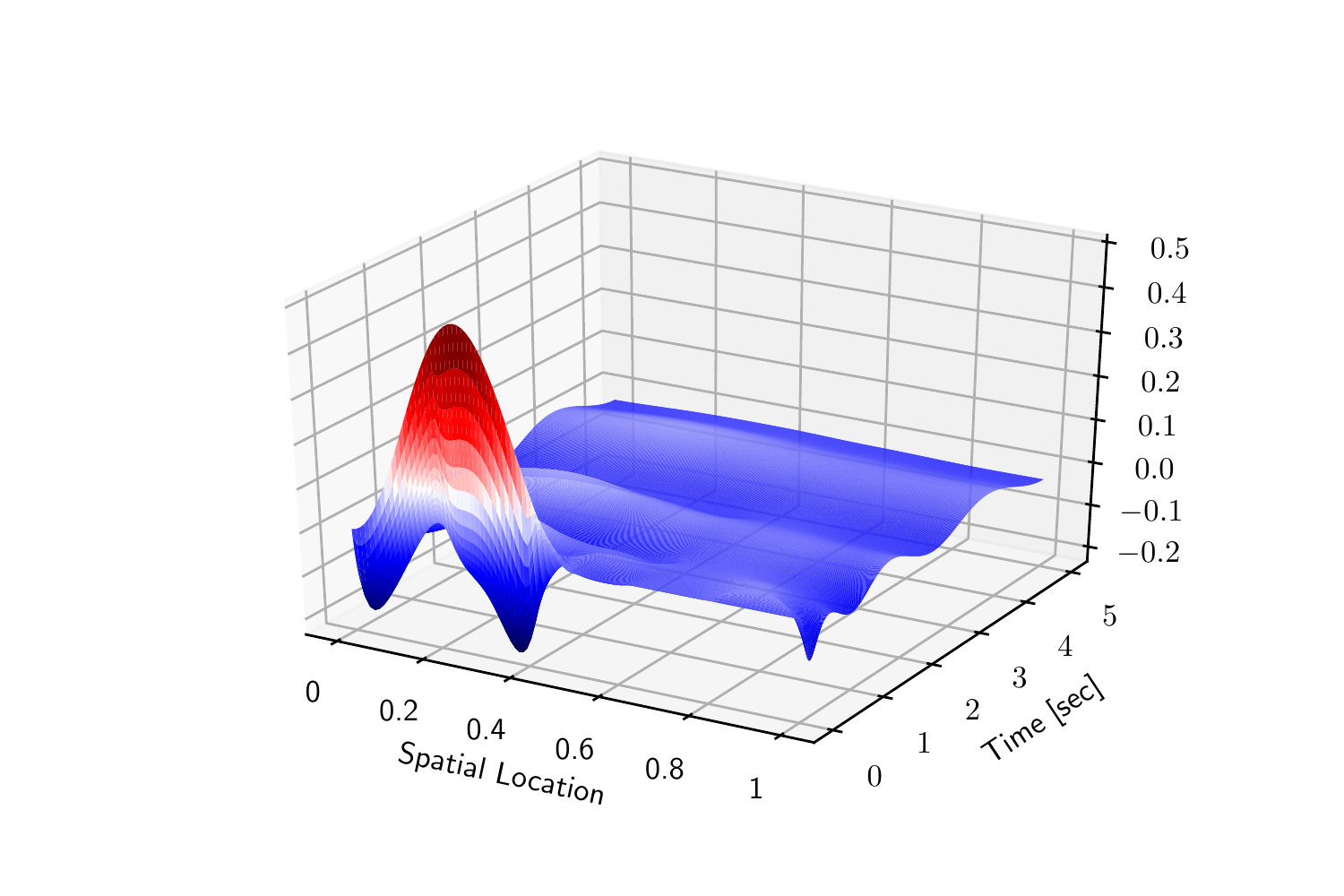}
		\caption{LQG-BT}
		\label{fig:FOMwLQGcontrol}
	\end{subfigure}
	\caption{FOM simulation: initial condition $x_0$, controlled by $u_r^{\text{BT}}$ and $u_r^{\text{LQGBT}}$ with $r=5$.}
\end{figure}

Finally, while the above results were obtained for $r=5$, Table~\ref{tbl:cost} gives an overview of results for different dimensions  $r$ of the nonlinear ROMs. We note that while there is an anomaly at $r=9$ for the BT-ROM-based controller, in general, BT-ROM controllers seem to outperform the LQG-BT ROM controllers, evidenced by a lower number of iterations, as well as FOM cost, which appears to be a result of the slow singular value decay of the LQG-BT singular values in Fig.~\ref{fig:Hankel_svals}.

\begin{table}[tb]
\caption{Comparison of cost of ROM controllers on FOM, $J(x, u_r)$, on ROM, $J(x_r, u_r)$, and number of iterations of ILQR for increasing dimension of the ROM.}
\label{tbl:cost}
\begin{center}
\setlength{\tabcolsep}{4pt} % Trick to reduce space betwee columns
\begin{tabular}{r|r|r|r|r|r|r }
 &     LQG-BT        &   LQG-BT     &       BT            &   BT                 &  LQG-BT   &    BT\\
$r$  &   $J(x_r,u_r)$ &   $J(x, u_r)$&  $J(x_r, u_r)$ &  $J(x, u_r)$   &  iter. &   iter.\\
\hline 
2  &                 			 	56.8              &   		197.3		       &				85.9       &    	245.0         &  		89		    &				58 \\
3  &                 				71.2               &  		136.6		       &				44.5       &    	101.8          & 		  58		&    			278 \\
4	&                 				45.8		      &     	89.8		       &				46.6      &     	83.0           & 		170         &          		277 \\
5  &                  				47.4		       &     	63.6           		&			53.8          & 	68.9            	&	170             &      		222 \\
6	&	             				50.5              &    		69.0            	&			44.3          &	 47.7	     		  & 1328               &    		966 \\
7	&                  				50.8              &    		61.3		    	&			43.9          & 	46.5           		&   1208           &        		861 \\
8   &                 				54.2              &    		63.3	          & 				43.5      &     	45.7	     &		  3647        &           		930 \\
9	&	              				53.3              &    		59.1          	&			272.7          &	272.1          		 & 9070               &    		101 \\
10 &                  				53.8             &     		58.8           	&				51.5         &  	51.2               &   2395           &        		2687 \\
11  &                 				60.4              &    		60.2            &				48.8        &   	48.7               &    7927          &         		4165 \\
\end{tabular}
\end{center}
\end{table}

%%%%%%%%%%%%%%%%%%%%%%%%%%%%%%%%%%%%%%%%%%%%%%%%%%%%%%%%%
\section{CONCLUSIONS AND FUTURE DIRECTIONS} \label{sec:conclusions}
%%%%%%%%%%%%%%%%%%%%%%%%%%%%%%%%%%%%%%%%%%%%%%%%%%%%%%%%%
We presented a new control framework for high-dimensional nonlinear systems by combining the iterative linear quadratic regulator (ILQR) framework and system-theoretic model reduction. As shown in the numerical experiments, both standard balanced truncation model reduction and LQG balanced truncation allowed us to compute well-performing controllers from ROMs with about $95\%$ reduction of model order compared to the full-order model.
Further work may be done by using nonlinear system-theoretic model reduction, such as \cite{benner2018mathcalh_2, antoulas2019loewner, KW2019_balanced_truncation_lifted_QB}, which despite being open-loop techniques, can provide more accurate ROMs for nonlinear systems. One would also need to investigate the observed correlation between larger state dimensions and increased iteration numbers for ILQR to understand scalability of ILQR. Lastly, convergence of ILQR for increasing $N$ remains an open problem.


%%%%%%%%%%%%%%%%%%%%%%%%%%%%%%%%%%%%%%%%%%%%%
\section*{ACKNOWLEDGMENT}
\noindent 
B.K. thanks Prof. Hermann Mena for discussions on ILQR.

%%%%%%%%%%%%%%%%%%%%%%%%%%%%%%%%%%%%%%%%%%%%%%%%%%%%%%%%%%%%%%%%%%%%%%%%%%%%%%%%
\bibliographystyle{IEEEtran}
\bibliography{bibliography_ILQR.bib}

\end{document}